\documentclass[12pt]{article}
\makeindex
\usepackage{amsfonts,amssymb,amsmath,amscd,amsthm,xcolor,mathtools,enumerate}

\usepackage{latexsym}
\usepackage{graphicx}
\usepackage[backref=false,colorlinks, linkcolor=blue, citecolor=red]{hyperref}

\newcommand{\Nrd}{{\rm Nrd}}
 
\newcommand{\disc}{{\rm disc}} 
\newcommand{\Disc}{{\rm Disc}} 

\newcommand{\ind}{{\rm ind}}

\newtheorem{theorem}{Theorem}[section]
\newtheorem{proposition}[theorem]{Proposition}  
\newtheorem{corollary}[theorem]{Corollary}  
 
\newtheorem{lemma}[theorem]{Lemma}

\numberwithin{equation}{section}

\title{ Rational equivalence on adjoint groups of type $^{1}D_n$ over field $\mathbb{Q}_P(X)$} 
\author{M. Archita} 
\date{ } 

\begin{document} 
	
	\maketitle 
	
	\begin{abstract}Let $F$ be the function field of a smooth, geometrically integral curve over a $p$-adic field with $p\neq 2.$ Let $G$ be a classical adjoint group of type $^1D_n$ defined over 
		$F$. We show that $G(F) / R$ is trivial, where $R$ denotes {\it rational equivalence} on $G(F)$. 
	\end{abstract}

	\section{Introduction} 
	Let $E$ be a field and G be an adjoint classical group. Let $G(E)$ be the group of $E$ rational points of G. In (\cite[Chapter II, \S$14$]{Ma}) Manin defined the notion of 
	{\it rational equivalence} on $G(E)$; in fact, it has been defined there more
	generally for a variety defined over a field. The set of equivalence 
	classes for this relation is denoted by $G(E) / R$ and has a natural group structure. It is a birational invariant of $G$ and the triviality of $G(E)/ R$ is closely related to the rationality of the underlying group variety $G$ over $E$. In fact, if 
	$G$ is $E$-rational then $G(E) / R$ is trivial 
	(see \cite{CTS}, \cite[\S$1$, Proposition $1$]{Me1}). 
	\medskip 
	
	Let $F$ be the function field of a smooth, geometrically integral curve over a $p$-adic field with $p\neq 2.$ Let $G$ be an absolutely simple adjoint classical group defined over $F.$ We are interested in the triviality of $G(F)/R.$ By the well known theorem of Weil, such a group $G$ is classified into types of $A_n,B_n,C_n,$ or $D_n$ ($D_4$ non-trialitarian), $n\geq 2$. We list below the known results in this direction:
	\begin{enumerate}
		\item If $G$ is of type $^1A_{n-1}, n\geq 2,$  $^2A_{n-1}$ with $n$ odd and $n\geq 3$, $B_n, n\geq 1,$ $G(F)/R$ is trivial due to Merkurjev. \cite{Me1}.)
		\item If $G$ is of type $C_n,$ $n=2,$ and $n$ odd  then $G(F)/R$ is trivial Merkurjev (see \cite{Me1}) and for other cases due to R. Preeti, and A. Soman. (see \cite{PS})
		
		\item If $G$ is of type $D_n,$ in \cite{PS}, it is shown that  $G(F)/R$ is trivial when the associated hermitian form $h$ has even rank, trivial discriminant, and trivial Clifford invariant.
		\item An example has been constructed to show $G(F)/R\neq 1,$ where $G$ is an adjoint group of type $^2D_3,$ and the associated hermitian form $h$ has even rank, non-trivial discriminant in \cite{PS}.
	\end{enumerate}	
	Let $(A,\sigma)$ be a central simple algebra with an orthogonal involution over $F$ associated to $G.$
	Further, we are interested in the behavior of $G(F)/R,$ when $(A,\sigma)$ has trivial discriminant, and non-trivial Clifford invariant. In this work, we establish the following:
	
	\begin{theorem}\label{Main Theorem}
		Let $F$ be the function field of a smooth, geometrically integral curve over a $p$-adic field with $p\neq 2.$ Let $G$ be a classical adjoint group of type $^1D_n$ ($D_4$ non-trialitarian) over $F.$ Then the group of rational equivalence classes over $F$ is trivial $i.e.$ $G(F)/R=(1).$
	\end{theorem}
	%\begin{theorem}
	%		Let $F$ be the function field of a smooth, geometrically integral curve over a $p$-adic field with $p\neq 2.$ Let $(A,\sigma)$ be a central simple algebra of degree $2m,$ where $m$ is an odd integer over $F,$   with an orthogonal involution $\sigma$ such that $\disc(\sigma)=1$. Let $h$ be a hermitian form over $(A,\sigma)$ of trivial discriminant. % trivial Rost invariant. %Additionally, consider that $\ind(C(A,\sigma))$ is square-free.
	%		Then $\mathbf{PSim}_+(A,\sigma)(F)/R=(1).$
	%	\end{theorem}
%	Further, when $\deg(A)=2m,$ where $m$ is an odd integer, we consider the rost invariant of the associated hermitian form is trivial.
%		\begin{theorem}
	%	Let $F$ be the function field of a smooth, geometrically integral curve over a $p$-adic field with $p\neq 2.$ Let $(A,\sigma)$ be a central simple algebra of degree $2m,$ where $m$ is an even integer over $F,$ with an orthogonal involution. Let $h$ be a hermitian form over $(A,\sigma)$ of trivial discriminant, trivial Rost invariant. 
	%	Then $$G_+(A,\sigma)=Hyp(A,\sigma)\cdot F^{*2}.$$
	
	%\end{theorem}
	\section{Preliminaries}
	Let $E$ be a field of characteristic different from $2$. 
	For a separable closure $E_s$ of $E$, let 
	$\Gamma_{E}$ be the Galois group $Gal(E_{s}/E)$. 
	Let $p$ be a prime integer. Recall that the $p$-cohomological 
	dimension of $\Gamma_E$ is said to be $\leq r$ if 
	$H^n (\Gamma_E, A) = 0 $ for all $n > r$ and for all finite 
	$p$-primary $\Gamma_E$ modules $A$. This is denoted as 
	$cd_p (\Gamma_E) \leq r$. 
	The {\it cohomological dimension} of $E$, 
	denoted by $cd(E)$, is said to be $\leq r$, if 
	$cd_p (\Gamma_E) \leq r$ for all prime integers $p$. 
	%A field $E$ is of {\it virtual cohomological dimension} 
	%	$vcd(E)$ equal to $r$, 
	%	if $cd (\Gamma_{E(\sqrt{-1})} )$ is $r$. 
	%	For example, totally imaginary number fields and 
	%	function fields of surfaces over ${\mathbb C}$ have cohomological dimension $\leq 2$. On the 
	%	other hand, number fields with atleast one real 
	%	completion and function fields of surfaces over 
	%	${\mathbb R}$ have virtual cohomological dimension 
	%	$\leq 2$. 
	\medskip

	In this paper we consider the base field $F$ be the function field of a smooth, geometrically integral curve over a $p$-adic field with $p\neq 2$ which satisfy 
	$cd (F) \leq 3$. Let $G$ be an absolutely simple classical adjoint group of type $^1D_n$ ($D_4$ non-trialitarian) over $F.$ Then by Weil's classification, (see \cite{We}) $G\cong\mathbf{PSim}_+(A,\sigma),$ where $(A,\sigma)$ is a central simple algebra of even degree with an orthogonal involution over $F.$ %The main theorem mentioned in the introduction is 
	%already produce a counter-example when $\disc(\sigma)\neq 1$ (see \cite[\S $7$]{PS1}). 
	%So we restrict to the case when $vcd (F) \neq cd (F)$. 
	%By \cite[Chapter II, \S$4.1$, Proposition 
	%$10^{\prime}$]{Se} this implies that the characteristic of 
	%$F$ is $0$ and that $F$ can be ordered. 
	%Let $\Omega_F$ be the set of orderings of $F$. For $v \in \Omega_F$, 
	%	$F_v$ denotes the real closure of $F$ at $v$. As is usual, 
	%	for $x, y \in F_v$, we write $x <_v y$ when $y -x$ is positive for the given ordering $v$ on $F_v$. 
	\medskip

	Let $(A, \sigma)$ be a central simple algebra over $E$ with an orthogonal involution. So $A$ is a form for the matrix algebra, i.e., 
	if $E_s$ is a separable closure of $E$ then 
	$A \otimes_E E_s \cong M_r(E_s)$ for some natural number $r$. Further, 
	$\sigma$ being an orthogonal involution means 
	$\sigma \otimes E_s$ corresponds to a symmetric form on $M_r (E_s)$. 
	We refer to \cite[\S$1A$ and \S$2A$]{KMRT} for details. 
	As $A$ is a form for the matrix algebra over $E_s$, the dimension of $A$ over $E$ is a square. The positive square root of dimension $A$ over $E$ is called the {\it degree} 
	of $A$. 
	Recall that under the above isomorphism of $E$ algebras 
	$A \otimes_E E_s \cong M_r(E_s)$, the determinant map on $M_r(E_s)$ descends 
	to a map $Nrd: A \to E$ called the {\it reduced norm}. 
	By the Wedderburn structure theorems, 
	there exists a central {\it division} algebra $D$ over $E$ such that 
	$A \cong M_n(D)$. Further, there exists an orthogonal involution 
	$\tau$ on $D$ and a hermitian form $h$ over $(D, \tau)$ such that 
	$(A, \sigma) \cong (M_n(D), \tau_h)$, where $\tau_h$ denotes the adjoint 
	involution (see \cite[\S$1A$ and \S$4A$]{KMRT}) for details. A central simple 
	algebra with an orthogonal involution $(A, \sigma)$ over 
	$E$ is said to be {\it hyperbolic} if $A$ contains an 
	idempotent $e$ such that $\sigma (e) = 1-e$. When 
	characteristic $E \neq 2$, 
	this is equivalent to the existence of an isotropic ideal 
	$I \subset A$ of dimension $\dim_E I = \frac{1}{2} 
	\dim_E A$. 
	\medskip 

	Let $q$ be a quadratic form over $F$ and let $V$ be its underlying vector space. The dimension of $q$ is the dimension of the vector space $V$ over $F$. Let $b_q$ be the associated bilinear form to $q$. If $b_q$ is non-degenerate then $q$ is said to be {\it non-degenerate} or {\it regular}. In this paper, by a quadratic form we mean a non-degenerate quadratic form. A quadratic form $q$ of dimension $n$ is equivalent to a diagonal form $a_1 X_1^2 + a_2 X_2^2 + \cdots + a_n X_n^2$, 
for some $a_i \in F$. The diagonal form $\sum_{i =1}^n 
a_i X_i^2$ is denoted by $\langle a_1, \cdots , a_n \rangle$. 
A quadratic form $q$ is said to be {\it isotropic }, if there exists a non zero vector $v \in V$ such that $q(v) = 0$. 
\medskip 
	
	A quaternion $E$ algebra $\left ( \frac{a,b}{E} \right )$ is an $E$ algebra of degree $2$, generated by $i, j$ with the 
	relations: $i^2 = a$, $j^2 = b$ and $ij = -ji$, where 
	$a, b \in E^* = E - \{ 0 \}$. The quadratic form $\langle 
	1, -a, -b, ab \rangle$ is called the {\it norm} form of 
	$\left ( \frac{a,b}{E} \right )$. It is well known that 
	$\left ( \frac{a,b}{E} \right )$ is a division algebra if and only if the associated norm form is anisotropic (see 
	\cite[Chapter $3$]{L} for details). For any $a,b,c,d\in E^*,$ the quadratic form $\langle a,b,-ab,-c,-d,cd\rangle$ is an Albert form of the biquaternion algebra $\left ( \frac{a,b}{E} \right )\otimes \left ( \frac{c,d}{E} \right ).$

	\medskip 
	
	%Let $F$ be a field as above, so that $vcd (F) \leq  2$. 
	A central simple algebra $A$ over $F$ is said to be {\it split} if it is isomorphic to the matrix algebra over $F$. 
	%	A central simple algebra $A$ over $F$ is said to be {\it locally split} 
	%	if $A_v := A \otimes_F F_v $ is split for every $v \in \Omega_F$. 
	%Let $(A, \sigma)$ be a central simple algebra over $F$ with an 
	%	orthogonal involution. 
	%	A hermitian form $h$ over $(A, \sigma)$ is said to be 
	%	{\it locally hyperbolic} if it is so over $F_v$ for every 
	%	$v \in \Omega_F$. 
	\medskip 
	
	We next recall some Galois cohomological invariants of hermitian 
	forms over $(D, \tau)$, where 
	$D$ is a central division algebra over $E$ with an orthogonal involution. 
	We refer the reader to \cite[\S$2$]{BP1} and \cite[\S$3.4$]{BP2} for details. Let $(V,h)$ 
	be a non-degenerate hermitian form over $(D, \tau)$. 
	%For a field $E$, we denote by $\mu_n(E)$ the group of $n$-th roots of unity in $E$. 
	%For any integer $i \geq 1$, if $n$ divides $m$, then there is a natural 
	%injection $\mu_n^{\otimes i} (E) \to \mu_m^{\otimes i} (E)$. The direct limit of this system, for $n$ coprime to the characteristic of $E$, is denoted by ${\mathbb Q} / {\mathbb Z} (i) (E)$. 
	
	The $rank (h)$ is 
	defined as the dimension of the underlying $D$-vector space $V$. 
	Given a basis of the $D$-module $V$, the hermitian form $h$ is given by some 
	matrix $M(h)$ in this basis. Let $A = M_n (D)$ and $m^2 = dim_E (A)$. 
	For an even rank hermitian form $h$, 
	the {\it discriminant}, $\disc (h)$ of $(V,h)$ is defined as 
	$\disc (h) = (-1)^{m(m-1)/2} \, \cdot \Nrd_A (M(h)) \in E^* / E^{*2}$. 
	We also need a refinement of the $\disc $ map called the {\it Discriminant} and which is denoted by $\Disc$. Let $\Disc (h) = 
	(-1)^{m(m-1)/2}\,  \Nrd (M(h)) \in F^* / (\Nrd (D)^*)^2$. 
	
	\medskip 
	
	Let $(V,h)$ be a hermitian form of rank $2n$ and trivial discriminant. 
	Let $H_{2n}$ be a hyperbolic form of rank $2n$ over $(D, \tau)$. 
	Then the Clifford invariant $C(h) := 
	Cl_{H_{2n}} (h)=[C_+(M_{2n}(D),\tau_h)] \in {\ }_2 Br (E) / (D)$ 
	is the {\it relative Clifford invariant} $Cl_{H_{2n}} (h)$ as defined by Bartels 
	(see \cite{B}, \cite[\S$2$]{BP1} for details). 
	%If $D =E$ then this invariant is the usual Clifford invariant of the quadratic form $h$. 
	\medskip 
	
	We recall the Rost invariant from \cite[\S 31]{KMRT}. $$R(h)=(R(\zeta))\in H^3(F,\mu_2)/(F^*\cup(A)),$$
	where $R(\zeta)$ takes values in $H^3(F,\mu_4^{\otimes 2}).$
	\medskip

	%	We next recall the notion of signature of a quadratic form with respect to orderings. Let $q = \langle a_1, \cdots , a_n 
	%	\rangle$ be a quadratic form over $F$ and $v \in \Omega_F$. 
	%	Over the real closure $F_v$, let $r$ be the number of $a_i$'s 
	%	that are positive and $s$ be the number of $a_i$'s that are negative. Then the {\it signature} of $q$ at $v$, denoted as 
	%	$\sign_v (q)$, is equal to $r-s$. 
	%	\medskip 
	
	%	Let $q$ be a quadratic form over $F$. Let $Sn: SO(q) \to F^*/ F^{*2}$ be the spinor norm. Then $Sn(q) \subset F^*$ denotes the subgroup generated by $F^{*2}$ and the representatives of the square classes in the image of $Sn$. 

	%	\medskip 
	
	Let $W(F)$ be the Witt ring of 
	$F$ (see \cite[Chapter $2$]{L} for details). 
	We denote by $I(F) \subset W(F)$ the {\it fundamental ideal} of even dimensional quadratic forms over $F$. The 
	$n$-th power of the ideal $I(F)$ is denoted by $I^n (F)$. 
	%	The total signature map $\sign$ is defined by 
	%	$$ 
	%	\sign : W(F) \to \prod_{v \in \Omega_F} W(F_v) 
	%	$$ 
	%	$$ 
	%	q \mapsto ((\sign_v (q) )_v) . 
	%	$$
	%	The kernel of this map $ker (\sign)$ is denoted by 
	%	$W_t (F)$. Every element of $W_t (F)$ is $2$-primary torsion (see \cite[Chapter $VIII$, 
	%	\S$3$]{L} for details). 
	%	\medskip

	If $A$ is split then by definition $A \cong M_n (F)$, for some 
	positive integer $n$. If $\sigma$ is an orthogonal involution on $M_n (F)$ then $\sigma$ is the adjoint involution 
	$\sigma_q$, corresponding to a quadratic form $q$ over $F$ (\cite[Chapter $1$]{KMRT}). If $n$ is even then the above mentioned invariants of $(M_n(F), \sigma_q)$ match with the 
	classical invariants of the quadratic form $q$, namely, 
	the dimension of $q$, the discriminant $\disc (q)$ and the 
	usual Clifford invariant of $q$. For a quadratic form $q$ over $F$ we also have the notion of a Clifford algebra, $c(q)$ 
	(see \cite[Chapter $V$]{L}, 
	\cite[Chapter $9$]{Sc} for details). 
	\medskip
	
	Let $(A, \sigma)$ be a central simple algebra of degree $ 2m$ over $E$ with an 
	orthogonal involution. An element $a\in A^*$ is said to be a 
	{\it similitude} if $\sigma(a)a \in E^*$. 
	The similitudes of $(A, \sigma)$ form a group which is denoted by 
	$Sim(A, \sigma)$. We refer to 
	\cite[Chapter $3$]{KMRT} for the basics on similitudes. 
	The multiplier map $\mu: Sim(A, \sigma) \to E^*$ defined by
	$\mu(a) = \sigma(a)a$ is a group homomorphism. 
	The image of this map is denoted by $G(A, \sigma)$. 
	Let $\sigma$ be adjoint to a hermitian form $h$. Then $\lambda \in 
	G(A, \sigma)$ if and only if $\lambda \, \cdot h \cong h$. 
	\medskip 
	
	Let Hyp($A,\sigma$) be the subgroup of $F^*$ generated by the norms from
	all those finite field extensions of $F$, where the involution $\sigma$ becomes hyperbolic.
	\medskip
	
	Let ${\bf Sim} (A,\sigma)$ be the algebraic group whose $E$ rational points are given by $Sim(A, \sigma)$. 
	Let ${\bf Sim_{+}} (A, \sigma)$ be the connected component of identity of 
	${\bf Sim} (A,\sigma)$. 
	Let $Sim_+ (A, \sigma)$ be the $E$-rational points of ${\bf Sim_+} (A, \sigma)$. 
	The elements of $Sim_{+}(A, \sigma)$ are called 
	{\it proper similitudes}. The group $\mu(Sim_{+}(A, \sigma))$ is denoted by $G_{+}(A, \sigma)$. 
	\medskip 
	
	When $A$ is split $i.e.$ $A\cong M_{2m}(E).$ By Morita equivalence the involution $\sigma$ corresponds to a quadratic form $q$ of dimension $2m$ over $E.$ In the split case we set $G(q)=G_+(A,\sigma),$ and Hyp($q$)=Hyp($A,\sigma$). 
	\medskip
	
	The group ${\bf PSim} (A, \sigma)$ of projective similitudes has
	$PSim (A, \sigma) = Sim (A, \sigma) / E^*$ as its group of 
	$E$ rational points. 
	The connected component of identity of the group ${\bf PSim} (A, \sigma)$ 
	is denoted by ${\bf PSim_{+}} (A, \sigma)$. 
	\medskip 
	\section{Some known results}
	The following well known theorem due to 
	Merkurjev characterizes the 
	group of $R$-equivalence classes of the adjoint group 
	${\bf PSim}_+ (A, \sigma)$ in terms of the multipliers 
	of similitudes of $(A, \sigma)$. 
	\medskip

	\begin{theorem} (\cite[Theorem $1$]{Me1}) \label{merkurjev}
		With notation as above 
		$$ 
		{\bf PSim}_+ (A,\sigma)(F)/R \simeq  
		G_{+}(A,\sigma)/Hyp(A,\sigma) \cdot F^{*2} .  
		$$ 
	\end{theorem} 
%	\begin{proposition}(\cite[Proposition 8.5]{KMRT})\label{hyperbolic}
%		Let $\underline{\sigma}$ and  $\underline{\sigma}_0$ be the canonical involution on $C(V,q)$ and $C_0(V,q)$ respectively. The involutions $\underline{\sigma}$ and $\underline{\sigma}_0$ are hyperbolic if the quadratic space
%		$(V, q)$ is isotropic.
%	\end{proposition}
	We note down the next proposition, for a $3$-fold Pfister form over $F.$
	\begin{proposition}(\cite[Proposition 5.1]{PS})\label{pfister}
		Let $F$ be the function field of a smooth, geometrically integral curve over a $p$-adic field with $p\neq 2.$ Let $q$ be a 3-fold Pfister form over $F.$ Then $F^*=Hyp(q)=G(q).$
	\end{proposition}
	\begin{corollary}(\cite[Corollary 4.4]{PS})\label{nrd}
	Let $F$ be the function field of a smooth, geometrically integral curve over a $p$-adic field with $p\neq 2.$ Let $A$ be a central simple algebra over $F$ of index $\leq 4.$ For any $\lambda\in F^*,$ if $(\lambda)\cup (A)=0\in H^3(F,\mu_4^{\otimes 2})$ then $\lambda\in \Nrd(A^*).$
	\end{corollary}
	The following theorem is noted for a split algebra over $F$: 
	\begin{theorem}(\cite[Theorem 7.1]{PS}) \label{split}
		Let $F$ be the function field of a smooth, geometrically integral curve over a $p$-adic field with $p\neq 2.$ Let $q$ be a quadratic form over $F$ of even dimension and trivial discriminant. Then $G(q)=Hyp(q)\cdot F^{*2}.$
		
	\end{theorem}
	For a non-split central simple algebra $A$ with an orthogonal involution over $F,$ we note down the following theorem:
	\begin{theorem}(\cite[Theorem 7.2]{PS}) \label{non-split}
		Let $F$ be the function field of a smooth, geometrically integral curve over a $p$-adic field with $p\neq 2.$ Let $(A,\sigma)$ be a central simple algebra over $F$ with an orthogonal involution. Let $h$ be a hermitian form over $(A,\sigma)$ of even rank $2n,$ trivial discriminant and trivial Clifford invariant. Then $$G_+(h)=Hyp(h)=F^{*}.$$
		
	\end{theorem}
	
	%	\begin{theorem}(\cite[Corollary 1.4]{P})\label{symp}
		%		Let $F$ be the function field of a smooth, geometrically integral curve over a $p$-adic field with $p\neq 2.$ Let $(A,\sigma)$ be a central simple algebra over $F$ of square-free index, with a unitary $F/K$ involution, where $K$ denotes the fixed field of $\sigma$ in $F.$ Let $h$ be a hermitian form over $(A,\sigma)$ such that $rank(h)$ is even, $\disc(h)$ is trivial in $K^*/N_{F/K}(F^*)$ and the Rost invariant $R(h)=0\in H^3(K,\mathbb{Q}/\mathbb{Z}(2))/Cor_{\mathbb{Q}_p/K}(F^*\cup (A)).$ Then $h=0\in W(A,\sigma).$
		%	\end{theorem}
	The three results recorded below are over a function field in one variable over a non-dyadic $p$-adic field.
	\begin{theorem}(\cite[Corollary 2.2]{VP}) \label{ind}	Let $k$ be a non-dyadic $p$-adic field and $K$ a function field in one variable over $k.$ 
		Let $D$ be a central division algebra over $K$ of exponent 2 in the	Brauer group of $K.$ Then the degree of $D$ is atmost 4.% In particular, every element in $H^2(K,\mathbb{Z}/2)$ is a sum of two symbols.		
	\end{theorem}
	\begin{theorem}(\cite[Theorem 4.5]{VP}) \label{rank}
		Let $k$ be a non-dyadic $p$-adic field and $K$ a function field in one variable over $k.$ Then every quadratic form over $K$ of rank at least 11 is isotropic.		
	\end{theorem}			
The following theorem states that $H^3(F,\mu_2)$ consists of symbols.	
	\begin{theorem}(\cite[Theorem 3.9]{VP}) \label{symbol}
		Let $k$ be a non-dyadic $p$-adic field and $K$ a function field in one variable over $k.$ Let $k$ be a non-dyadic $p$-adic field and $K$ a function field in one variable over $k.$ Let $\alpha_i\in H^3(K,\mu_2), 1\leq i \leq n.$ Then there exists $a,b,c_i\in K^*$ such that $\alpha_i=(a)\cdot (b)\cdot(c_i).$ In particular, every element in $H^3(K,\mu_2)$ is a symbol.
	\end{theorem}	
	\section{Main theorems}	
	%\begin{theorem}
	%		Let $F$ be the function field of a smooth, geometrically integral curve over a $p$-adic field with $p\neq 2.$ Let $(A,\sigma)$ be a central simple algebra over $F$ with an orthogonal involution. Let $h$ be a hermitian form over $(A,\sigma)$ of even rank $2n,$ trivial discriminant. Then $$G_+(h)=Hyp(h)=F^{*}.$$
	
	%	\end{theorem}
	
\subsection {\bf Case $\deg(A)=2m,$ where m is an odd integer}
Let $F$ be the function field of a smooth, geometrically integral curve over a $p$-adic field with $p\neq 2.$ Let $(A,\sigma)$ be a central simple algebra of degree $2m,$ where $m$ is an odd integer over $F,$   with an orthogonal involution $\sigma.$ If $A$ is split then it is shown that $\mathbf{PSim}_+(A,\sigma)(F)/R=1$ in \cite[Theorem 7.1]{PS}. If $A$ is non-split then, $A\cong M_m(Q),$ for some quaternion division algebra $Q.$ Further, there exists an orthogonal involution $\tau$ on $Q$ such that $\sigma=\tau_h$, for
some hermitian form $h$ of odd rank $m$ over $(Q, \tau),$ i.e., $(A, \sigma)
\cong
(M_m(Q), \tau_h)$ (see
\cite[Chapter 1]{KMRT}). We have the following theorem.
\begin{theorem}\label{odd}
	Let $F$ be the function field of a smooth, geometrically integral curve over a $p$-adic field with $p\neq 2.$ Let $(A,\sigma)$ be a central simple algebra of degree $2m,$ where $m$ is an odd integer over $F,$   with an orthogonal involution $\sigma$ such that $\disc(\sigma)=1$.% Let $h$ be a hermitian form over $(A,\sigma)$ of trivial discriminant.% trivial Rost invariant. %Additionally, consider that $\ind(C(A,\sigma))$ is square-free.
	Then $$G_+(A,\sigma)=Hyp(A,\sigma)\cdot F^{*2}.$$
	
\end{theorem}
\begin{proof}
	Let $A$ be a central simple algebra of $\deg(A)= 2m$, where $m=2r+1$ is an odd integer. For $r=0,$ the algebra $A$ becomes a quaternion division algebra and quaternion division algebras do not carry any orthogonal involution with trivial discriminant (see \cite[\S 7]{KMRT}). Now
	$\exp(A)$ divides $\ind(A),$ and $\ind(A)$ divides $\deg(A).$ Hence, $\ind(A)\in\{1,2\}.$ So $D_A=F,$ or $D_A$ is a quaternion division algebra over $F.$ Hence, $\Nrd_{D_A}(D_A^*)$ contains $F^{*2}.$ Also, by \cite[Theorem 1, pg 146 ]{D}, $\Nrd_{D_A}(D_A^*)=\Nrd_A(A^*).$ Hence, $$F^{*2}\cdot \Nrd_A(A^*)\subseteq \Nrd_A(A^*)\subseteq F^{*2}\cdot \Nrd_A(A^*). $$
	%We will try to show $G_+(A,\sigma)=\Nrd(C(A,\sigma)^*)=Hyp(A,\sigma).$
	Consider $\lambda\in G_+(A,\sigma).$ It implies that, $\lambda=\sigma(a)\cdot a$ for some $a\in A^*.$ So, $\lambda^{2r+1}=\Nrd(a),$ and $\lambda= \Nrd(a)\cdot (\lambda^{-r})^2\in \Nrd_A(A^*).$ Hence, there exists a finite field extension $L/F$ such that $A_L=0,$ $\lambda\in N_{L/F}(L^*),$ and $\sigma_L$ corresponds to a quadratic form of even rank, trivial discriminant. %By Theorem \ref{split}, $G_+(A_L,\sigma_L)=Hyp(A_L,\sigma_L)\cdot L^{*2}.$
	%	Claim: $$G_+(A_L,\sigma_L)=Hyp(A_L,\sigma_L)\cdot L^{*2}=L^*.$$
	%Now, $A_L=0,$ and $\sigma_L$ is a quadratic form of even rank, trivial discriminant.
	%\medskip
	
%	Since $\deg(A_L)=4r+2,$ the canonical involution $\underline{\sigma}_L$ on $C_0(A_L,\sigma_L)$ is a unitary involution. %Now we will proceed depending on the rank of the quadratic form $\sigma_L.$
	According to our assumption, $\deg(A)\geq 6$ and hence, rank$(\sigma_L)\geq 36.$ As rank$(\sigma_L)\geq 11,$ using Theorem \ref{rank}, for  we obtain that {$\sigma$}$_L$ is isotropic. %Hence, $\Nrd(C(A_L,\sigma_L)^*)=L^*.$ So, by Lemma \ref{Hyp-Nrd(C)}, we obtain that $L^*=\Nrd(C(A_L,\sigma_L)^*)\subseteq Hyp(A_L,\sigma_L).$
	Hence, $Hyp(A_L,\sigma_L)=L^*$ and, it follows that, $\lambda\in Hyp(A,\sigma).$
\end{proof}
%	\medskip

%	By the result in \ref{ind},  observe that $\ind(C_0(A_L,\sigma_L))$ at most 4. Using Theorem \cite{P}, we obtain $\underline{\sigma}_L$ is hyperbolic. Hence, $\Nrd(C(A_L,\sigma_L)^*)=L^*.$ So, by Lemma \ref{Hyp-Nrd(C)}, we obtain that $Hyp(A_L,\sigma_L)=\Nrd(C(A_L,\sigma_L)^*)=L^*=G_+(A_L,\sigma_L).$ Hence, $\lambda\in Hyp(A,\sigma).$

%		\end{proof}
\subsection {\bf Case $\deg(A)=2m,$ where m is an even integer}
	Let $A$ be a central simple algebra of $\deg(A)=2\cdot2r,$ where $r$ is an integer. As, the exponent of $A,$ $\exp(A)\leq 2,$ we have $\ind(A)\leq 4,$ (see Theorem \ref{ind}). Hence, either in the non-split case, $(A,\sigma)\cong (M_{2r}(Q),\tau_h)$ for some quaternion division algebra $Q$ with an orthogonal involution $\tau,$ and $h$ is a hermitian form of even rank over $(Q,\tau)$ or, $(A,\sigma)\cong (M_r(H),\gamma_{h})$ for some biquaternion division algebra $H=Q_1\otimes Q_2$ with an orthogonal involution $\gamma,$ and $h$ is a hermitian form of any rank over $(H,\gamma).$ Moreover, $\deg(A)\equiv 0$ (mod $4$) and $\disc(\sigma)=1.$ Hence the Clifford algebra $C(A,\sigma)=C_+(A,\sigma)\times C_(A,\sigma)$ (\cite[\S 8]{KMRT}).
\medskip 

In this subsection we start with the following lemma:
\begin{lemma}\label{Hyp-Nrd(C)}
	Let $F$ be the function field of a smooth, geometrically integral curve over a $p$-adic field with $p\neq 2.$ Let $(Q,\sigma)$ be a quaternion division algebra over $F$ with an orthogonal involution $\sigma$ such that $\disc(\sigma)=1\in F^*/F^{*2}.$  Let $h$ be a hermitian form of even rank over $(Q,\sigma).$ Then $\Nrd(C_+(h))\subseteq Hyp(h).$
\end{lemma}

\begin{proof}

	We will start with the case when $A\cong M_{2r}(Q),$ for some quaternion division algebra $Q.$ Consider $\lambda\in \Nrd_{A}(C_+(A,\sigma)^*)=\Nrd(C_+(h)).$ Hence, there exist a finite field extension $L/F$ such that $C_+(h_L)=0,$ and $\lambda\in N_{L/F}(L^*).$ Hence, $h_L$ is a hermitian form over $(Q_L,\tau_L)$  of even rank, trivial discriminant and trivial Clifford invariant. So, by Theorem \ref{non-split}, $G_+(h_L)=Hyp(h_L)= L^*.$ Thus, we obtain that $\Nrd_A(C_+(h))\subseteq Hyp(h).$	
\end{proof}

\begin{lemma}\label{HE}
	Let $F$ be the function field of a smooth, geometrically integral curve over a $p$-adic field with $p\neq 2.$ Let $(Q_1\otimes Q_2,\sigma)$ be a biquaternion division algebra over $F$ with an orthogonal involution $\sigma$ such that $\disc(\sigma)=1\in F^*/F^{*2}.$  Let $h$ be a hermitian form of any rank over $(Q_1\otimes Q_2,\sigma).$ Then $\Nrd(C_+(h))\subseteq Hyp(h).$
\end{lemma}
\begin{proof}

Let $h$ be the hermitian form over $(Q_1\otimes Q_2, \sigma)$ with any rank, say $r.$ Consider, $\lambda\in \Nrd(C_+(h))= \Nrd(C_+(M_{r}(Q_1\otimes Q_2),\sigma_h))\subseteq F^*.$  Then there exist a finite extension $L/F,$ such that $C_+(h_L)=0$ and $\lambda=N_{L/F}(\alpha),$ for some $\alpha\in L^*.$ 
\medskip

On the other hand, let $q$ be an Albert form for $Q_1\otimes Q_2.$ So, $q$ is a $6$-dimensional quadratic form over $F$ with trivial discriminant. As the u-invariant of $F,$ $u(F)=8$ (see \cite[Theorem 4.6]{VP1}), the group of spinor norms of $q,$ Sn($q$)=$F^*.$ Hence, $\lambda\in Sn(q).$ Thus there exists a finite  extension $K/F$ such that $q_K$ is isotropic, and over $K,$ $Q_1\otimes Q_2\sim Q,$ for some quaternion division algebra $Q.$ By Morita correspondence $(M_{r}(Q_1\otimes Q_2),\sigma_h)$ will correspond to $(M_{2r}(Q),\sigma_{h_K}).$ Over the composite field $M=L(\sqrt{-\alpha})K,$ $C_+(h_M)=0.$ By Theorem \ref{non-split}, $Hyp(h_M)=M^*,$ and $\lambda= N_{L/F}N_{M/L}(\sqrt{-\alpha}).$ Thus, $\lambda\in Hyp(h).$
%Now, $\ind(C_+(M_{2r}(Q_1\otimes Q_2),\sigma_h))\leq 2.$
\end{proof}
\begin{theorem}\label{even}
Let $F$ be the function field of a smooth, geometrically integral curve over a $p$-adic field with $p\neq 2.$ Let $(A,\sigma)$ be a central simple algebra of degree $2m,$ where $m$ is an even integer over $F,$ with an orthogonal involution, Let the involution $\sigma$ has trivial discrimimant. %Let $h$ be a hermitian form over $(A,\sigma)$ of trivial discriminant. 
Then $$\mathbf{PSim}_+(A,\sigma)(F)/R=(1)$$

\end{theorem}
\begin{proof}
$\deg(A)= 2m$, where $m=2r$ is an even integer. Let $h$ be the associated hermitian form over $(D,\tau)$ where, $D$ is a central division algebra with an orthogonal involution $\tau,$ and $A$ is a matrix algebra over $D.$ As, $\exp(A)\leq 2$ by corollary \ref{ind}, $\ind(A)\leq 4.$ Then $D$ is some quaternion algebra $Q$ over $F$ or, some biquaternion division algebra $H$ over $F,$ and $(A,\sigma)\cong (M_{2r}(Q),\tau_h)$ or, $(A,\sigma)\cong (M_r(H),\tau_h)$ accordingly. 
\medskip

 The Rost invariant $R(h)$ can be defined and\\ $R(h)=(R(\eta))\in H^3(F, \mu_4^{\otimes 2})/(H^1(F,\mu_2)\cup (A)).$ As $R(\eta)\in H^3(F, \mu_4^{\otimes 2}),$ it implies that, $2\cdot R(\eta)_L\in H^3(F,\mu_2).$ As, $H^3(F,\mathbb{Z}/2)$ consists of symbols. So, there exists $a,b,c\in F^*$ such that $2\cdot R(\eta)=(a)\cup (b) \cup (c),$ where $(a), (b), (c)$ denote their respective classes in $F^*/F^{*2}.$ Consider the $3$-fold Pfister form $q=\langle 1, -a \rangle\cdot \langle 1, -b \rangle\cdot \langle 1, -c \rangle \in I^3(F).$  Now, Proposition \ref{pfister} implies that, $F^*=G(q)=Hyp(q).$
\medskip

Consider $\lambda\in G_+(h)\subseteq F^*=G(q)=Hyp(q).$ Hence, there exists a finite field extension $L/F$ such that $q_L=0$ and $\lambda= N_{L/F}(\alpha),$ for some $\alpha\in L^*.$ Hence, $2\cdot R(\eta)_L=0$ in $H^3(L, \mu_2)$ and $R(\eta)\in  H^3(L,\mu_2).$  As, $H^3(L,\mu_2)$ consists of symbols, there exits $e,f,g\in L^*$ such that $R(\eta)_L=(e)\cup (f) \cup (g),$ where $(e), (f), (g)$ denote their respective classes in $L^*/L^{*2}.$ Consider the $3$-fold Pfister form $q_1=\langle 1, -e \rangle\cdot \langle 1, -f \rangle\cdot \langle 1, -g \rangle \in I^3(L).$  Now, Proposition \ref{pfister} implies that, $L^*=G(q_1)=Hyp(q_1).$ So, there exist a finite extension $M/L$ such that $(q_1)_M=0$ and $\alpha= N_{M/L}(\beta),$ for some $\beta\in M^*.$ Hence, $R(\eta)_M=0.$  Also note that $\lambda= N_{L/F}(\alpha),$ for some $\alpha\in L^*,$ and $\alpha= N_{M/L}(\beta),$ for some $\beta\in M^*.$
\medskip

On the other side, as $\lambda\in G_+(h),$ we have, $\langle 1, -\lambda \rangle\cdot h$ is hyperbolic and $R(\langle 1, -\lambda \rangle\cdot h)=0$ in $H^3(F,\mathbb{Z}/2)/(F^*\cup [A]).$ Hence, $(\lambda)\cup [C_+(h)]=(x)\cup [A],$ for some $x\in F^*.$ As $R(h)_M=0,$ we have %$(x)\cup [A_M]=0$ in $H^3(M, \mu_4^{\otimes 2})$ and so 
$(\beta)\cup [C_+(h)_M]=0$ in $H^3(F, \mu_4^{\otimes 2}).$  As $\exp(C_+(h))= 2,$ Theorem \ref{ind}, tells that $\ind(C_+(h))\leq 4.$ Thus using Corollary \ref{nrd}, we obtain that $\beta\in \Nrd(C_+(h)_M).$ By combining the above Lemma \ref{Hyp-Nrd(C)}, alongwith Lemma \ref{HE} it follows that $\beta\in Hyp(h_M).$ So, $\lambda=N_{L/F}(N_{M/L}(\beta)).$ Hence, we obtain that $\lambda\in Hyp(h).$
\end{proof}
%\begin{proof}

%		$\deg(A)= 2m$, where $m=2k$ is an even integer. By Lemma \ref{Hyp-Nrd(C)}, we have $Hyp(A,\sigma)=\Nrd(C(A,\sigma)),$ and $Hyp(A,\sigma)\subseteq \Nrd(A^*).$ 
%\medskip

%	Let $h$ be the associated hermitian form of $(A,\sigma)$ having even rank. The rost invariant $R(h)$ can be defined and\\ $R(h)=(R(\eta))\in H^3(F, \mu_4^{\otimes 2})/H^1(F,\mu_2)\cup (A).$ It implies, $2\cdot R(\eta)\in H^3(F,\mu_2).$ As, $H^3(F,\mathbb{Z}/2)$ consists of symbols. So, there exists $a,b,c\in F^*$ such that $2\cdot R(\eta)=(a)\cup (b) \cup (c),$ where, $(a), (b), (c)$ denote their respective classes in $F^*/F^{*2}.$ Consider the $3$-fold Pfister form $q=\langle 1, -a \rangle\cdot \langle 1, -b \rangle\cdot \langle 1, -c \rangle \in I^3(F).$  Now, Proposition \ref{pfister} implies that, $F^*=G(q)=Hyp(q)\subseteq Hyp(h).$
%	\medskip

%	Next consider $\lambda\in Hyp(q).$ Hence, there exists a finite field extension $L/F$ such that $q_L=0$ and $\lambda\in N_{L/F}(L^*).$ Using Proposition \ref{hyperbolic} we obtain that, $\underline{\sigma}_L$ is hyperbolic. Hence, using \ref{Hyp-Nrd(C)}, we obtain $Hyp(C(A_L,\sigma_L), \underline{\sigma}_L)=L^*\subseteq \Nrd(C(A_L,\sigma_L)^*).$ 
%	\medskip

%	So, $L^*=\Nrd(C(A_L,\sigma_L))=Hyp(A_L,\sigma_L).$ Thus, $\lambda\in Hyp(h).$  Hence, $Hyp(q)\subseteq Hyp(h),$ and further ./we obtain that, $G(h)=F^*=Hyp(h).$

%\end{proof}
\begin{proof}[proof of Theorem \ref{Main Theorem}]
A classical adjoint group $G$ is a direct product of groups $R_{L_{i}|F}(G_{i}),$ where $L_{i}|F$ are finite field extensions, $R_{L_{i}|F}$ is the Weil restriction and $G_{i}$ are absolutely simple classical adjoint groups defined over $L_{i},$ (\cite{T}). Further, $G_{i}(L_{i})/R\cong R_{L_{i}|F}(G_{i})(F)/R$ and $R$-equivalence commutes with direct products (\cite{CTS}, pg. 195). Hence it is sufficient to prove the theorem for an absolutely simple classical adjoint group $G$ defined over $F$. By \cite{We} such a $G$ is isomorphic to $\mathbf{PSim}_{+}(A,\sigma),$ for a central simple algebra $A$ with an involution $\sigma$ of either kind over $F.$ The proof now follows by combining the results from Theorem \ref{odd}, and Theorem \ref{even} from this paper.  
\end{proof}
\hspace{50 mm}{\bf Acknowledgments}
\newline

The author is supported by a post-doctoral fellowship provided by the National Board for Higher Mathematics, India, (File no: 0204/27/(5)/2023/R\&d-II/1183), during this work. She expresses her gratitude to Maneesh Thakur for making attention to this question and also for their productive conversations during this work. She would like to thank the Statistics and Mathematics Unit, Indian Statistical Institute, Bangalore center, India  for providing a wonderful research atmosphere and resources.

\bigskip 

\noindent M. Archita\\
E-mail: {\ttfamily architamondal40@gmail.com} \& {\ttfamily archita\_pd@isibang.ac.in}\\
Theoretical Statistics and Mathematics Unit (SMU),\\ Indian Statistical Institute, Bangalore center,\\ Bangalore, Karnataka-560059, India
\vspace{.5cm}

%	\noindent R. Preeti\\
%	E-mail: {\ttfamily preeti@math.iitb.ac.in}\\
%	Department of Mathematics, Indian Institute of Technology (Bombay),\\
%	Powai, Mumbai-400076, India

\end{document}